\documentclass[11pt,centertags]{amsart}
\usepackage{amscd}

\usepackage{easybmat}

\usepackage{mathrsfs}
\usepackage{amsfonts}
\usepackage{color}
\usepackage{pifont}
\usepackage{upgreek}
\usepackage{bm}
\usepackage{hyperref}
\usepackage{shorttoc}

\usepackage{amsmath,amstext,amsthm,a4,amssymb,amscd}
\usepackage[mathscr]{eucal}
\usepackage{mathrsfs}
\usepackage{epsf}
\numberwithin{equation}{section}

\def\p{\partial}
\def\o{\overline}
\def\b{\bar}

\def\mc{\mathcal}
\def\n{\nabla}

\def\w{\wedge}
\def\m{\omega}

\def\p{\partial}
\def\o{\overline}
\def\b{\bar}

\def\mc{\mathcal}
\def\n{\nabla}

\def\w{\wedge}

\newtheorem{thm}{Theorem}[section]
\newtheorem{lemma}[thm]{Lemma}

\newtheorem{cor}[thm]{Corollary}
\theoremstyle{definition}
\newtheorem{rem}[thm]{Remark}
\newtheorem{ex}[thm]{Example}
\theoremstyle{definition}
\newtheorem{defn}[thm]{Definition}
\newcommand{\be}{\begin{eqnarray}}
\newcommand{\ee}{\end{eqnarray}}

\newcommand{\comment}[1]{}

\usepackage{fancyhdr}
\begin{document}

\title{Chern forms of holomorphic Finsler vector bundles and some applications}
\author[Huitao Feng]{Huitao Feng$^1$}
\author[Kefeng Liu]{Kefeng Liu$^2$}
\author[Xueyuan Wan]{Xueyuan Wan$^3$}
\address{Huitao Feng: Chern Institute of Mathematics \& LPMC,
Nankai University, Tianjin, China}

\email{fht@nankai.edu.cn}

\address{Kefeng Liu: Department of Mathematics, UCLA, Los Angeles, USA}

\email{liu@math.ucla.edu}

\address{Xueyuan Wan: Chern Institute of Mathematics \& LPMC,
Nankai University, Tianjin, China}

\email{xywan@mail.nankai.edu.cn}

\thanks{$^1$~Partially supported by NSFC (Grant No. 11221091, 11271062)}

\thanks{$^3$~Partially supported by NSFC (Grant No. 11221091)
and the Ph.D. Candidate Research Innovation Fund of Nankai University}

\begin{abstract}  In this paper, we present two kinds of total Chern forms $c(E,G)$ and $\mathcal{C}(E,G)$ as well as a total Segre form $s(E,G)$ of a holomorphic Finsler vector bundle $\pi:(E,G)\to M$ expressed by the Finsler metric $G$, which answers a question of J. Faran (\cite{Faran}) to some extent. As some applications, we show that the signed Segre forms $(-1)^ks_k(E,G)$ are positive $(k,k)$-forms on $M$ when $G$ is of positive Kobayashi curvature; we prove, under an extra assumption, that a Finsler-Einstein vector bundle in the sense of Kobayashi is semi-stable; we introduce a new definition of a flat Finsler metric, which is weaker than Aikou's one (\cite{Aikou}) and prove that a holomorphic vector bundle is Finsler flat in our sense if and only if it is Hermitian flat.
\end{abstract}
\maketitle

\setcounter{section}{-1}
\tableofcontents
\section{Introduction} \label{s0}
Let $M$ be a complex manifold and let $\pi:E\to M$ be a holomorphic vector bundle. A Finsler metric $G$
on $E$ is a continuous function $G:E\to\mathbb{R}$ satisfying the following conditions:

\begin{description}
  \item [F1)] $G$ is smooth on $E^o=E\setminus O$, where $O$ denotes the zero section of $E$;
  \item[F2)] $G(z,v)\geq 0$ for all $(z,v)\in E$ with $z\in M$ and $v\in\pi^{-1}(z)$, and $G(z,v)=0$ if and only if $v=0$;
  \item[F3)] $G(z,\lambda v)=|\lambda|^2G(z,v)$ for all $\lambda\in\mathbb{C}$.
\end{description}

A holomorphic vector bundle $E$ with a complex Finsler metric $G$ is called a holomorphic Finsler vector bundle. In applications, one often requires that $G$ is strongly pseudo-convex, that is,

\begin{description}
  \item[F4)] the Levi form ${\sqrt{-1}}\partial\bar\partial G$ on $E^o$ is positive-definite along fibres $E_z=\pi^{-1}(z)$ for $z\in M$.
\end{description}

Clearly, any Hermitian metric on $E$ is naturally a strongly pseudo-convex complex Finsler metric on it. Note that there is also a concept of complex Finsler vector bundle over a smooth manifold, but we will not discuss this topic in this paper.

A most notable feature of a Finsler metric $G$ on $E$ is the homogeneous condition {\bf F3)}, by which, the Levi form ${\sqrt{-1}}\partial\bar\partial G$ can be viewed naturally as a real $(1,1)$-form on total space $E^o$ of the fibration $\pi:E^o\to M$. Moreover, if $G$ is strongly pseudo-convex, then by {\bf F4)}, the Levi form ${\sqrt{-1}}\partial\bar\partial G$ induces naturally Hermitian metrics, which is simply denoted by $h^G$, on the pull-back bundle $\pi^*E$ as well as the tautological line bundle $\mathcal{O}_{P(E)}(-1)$ over the projectivized space $P(E)$ of $E$ (cf. \cite{Ko1}, \cite{Ko2}). Different from the Hermitian case, here the induced metrics $h^G$ on $\pi^*E$ and $\mathcal{O}_{P(E)}(-1)$ depends on the points in  the fibres $P(E_z)$, $z\in M$.

Essentially, the geometry of holomorphic Finsler vector bundles can be considered as an important part of Hermitian geometry, and is closely related to the study of algebraic geometry. Note that $\mathcal{O}_{P(E)}(-1)$ can be obtained from $E$ by blowing up the zero section of $E$ to $P(E)$, the manifold $\mathcal{O}_{P(E)}(-1)\setminus O$ is biholomorphic to $E^o$. From this Kobayashi pointed out explicitly in \cite{Ko1} there is a one to one correspondence  $G\leftrightarrow h^G$, which we will call the Kobayashi correspondence, between the complex Finsler metrics on $E$ and the Hermitian metrics on $\mathcal{O}_{P(E)}(-1)$. Starting from this important observation, S. Kobayashi in \cite{Ko1} obtained an equivalent characterization of the ampleness in algebraic geometry by using a special curvature of a strongly pseudo-convex Finsler metric. This characterization gives a way to study the ampleness, positivity, (semi-)stability of holomorphic vector bundles from the point view of complex Finsler geometry.

Through this paper we always assume that $M$ is a closed complex manifold of dimension $n$ and $\pi:E\to M$ a holomorphic vector bundle of rank $r$, and the Finsler metrics involved are strongly pseudo-convex.

In this paper we mainly aim to construct Chern forms and Segre forms of a holomorphic Finsler vector bundle $\pi:(E,G)\to M$ by using the Finsler metric $G$ on $E$ and to explore some applications of these differential forms.

For readers convenience, we will give a brief review of basic definitions, notations and facts in the geometry of holomorphic Finsler vector bundles in Section 1. Especially, we recall a kind of curvature $\Psi$ of a holomorphic Finsler metric first introduced by S. Kobayashi in an explicit way in \cite{Ko1}, which we will call the Kobayashi curvature. The Kobayashi curvature $\Psi$ plays a very important role in many aspects in the geometry of holomorphic Finsler vector bundles.

In Section 2 we begin with the study of J. Faran's following question (cf. \cite{Faran}) :

\emph{Is it possible to define Chern forms of a holomorphic vector bundle using complex Finsler structures?}

Our approach follows the usual way to deal with Chern classes in Hermitian geometry, that is, to work with the tautological line bundle $\mathcal{O}_{P(E)}(-1)$ over $P(E)$. More precisely, let $\mathcal{O}_{P(E)}(1)$ denote the dual line bundle of $\mathcal{O}_{P(E)}(-1)$ and let $\Xi=c_1(\mathcal{O}_{P(E)}(1),(h^G)^{-1})$ denote the first Chern form of $(\mathcal{O}_{P(E)}(1),(h^G)^{-1})$. The key point here is to make use of the closed $(r-1,r-1)$-form $\Xi^{r-1}$ to push-down closed forms on $P(E)$ to ones on $M$. In the process, the form $\Xi^{r-1}$ serves as a ``Thom form" of the holomorphic fibration $\pi:P(E)\to M$. In this way, we obtain two kinds of total Chern forms $c(E,G)$ and $\mathcal{C}(E,G)$ as well as a total Segre form $s(E,G)$ of $(E,G)$ expressed by the Finsler metric $G$ on $E$. Moreover, we also get a Bott-Chern transgression term between these two kinds of total Chern forms. Here we mention a related work \cite{IDA}, in which C. Ida defined a kind of horizontal Chern forms by using a partial connection on the vertical tangent bundle of $P(E)\to M$. Note that these forms are defined on the total space $P(E)$ and the relationship with the classical Chern forms is not clear yet, as the author pointed out in his paper. As an application of Chern forms and Segre forms constructed in this section, we show that the signed Segre forms $(-1)^ks_k(E,G)$ are positive $(k,k)$-forms on $M$ when $G$ has the positive Kobayashi curvature.
There are lots works in this direction. For examples, in \cite{BG}, S. Bloch and D. Gieseker proved that all Chern classes of an ample vector bundle $E$ are numerically positive, that is, the paring $\langle c_k(E),[N]\rangle>0$ for any subvariety $N$ in $M$. W. Fulton and R. Lazarsfeld \cite{Fulton} generalized this result of Bloch and Gieseker and proved that all the Schur polynomials in the Chern classes of an ample bundle $E$ are numerically positive, which solved a related conjecture of Griffith. Note that the signed Segre classes are particular Schur polynomials in the Chern classes, so all signed Segre classes are also numerically positive. In \cite{Guler}, Guler showed that if $(E,h)$ is Griffiths positive, then the signed Segre forms $(-1)^{k}s_{k}(E,h)$ are positive $(k,k)$-forms for $k=1,\cdots, n$. To give a pure differential geometric proof of Fulton-Lazarsfeld's result is still an interesting question.

The semi-stability of a holomorphic vector bundle is an algebro-geometric concept. For a holomorphic vector bundle over a K\"{a}hler manifold, Kobayashi \cite{Ko2} introduced a notion of Finsler-Einstein metric and asked that whether a Finsler-Einstein vector bundle is semi-stable or not. In Section 3, by using the first Chern form $\mathcal{C}_1(E,G)$, we prove, under an extra assumption, that a Finsler-Einsler vector bundle in the sense of Kobayashi is semi-stable. Note that in \cite{Sun}, Sun discussed some properties of ``Finsler-Einstein" vector bundles in the sense of T. Aikou. However, Aikou's definition seems to be too strong for application. In this section we also prove a Kobayashi-L\"{u}bke type inequality related to Chern forms $\mathcal{C}_1(E,G)$ and $\mathcal{C}_2(E,G)$ of a Finsler-Einstein vector bundle $(E,G)$, which generalizes a recent result of Diverio in the Hermitian case.

By using a partial connection, Aikou \cite{Aikou1}, \cite{Aikou} introduced a notion of flatness for a holomorphic Finsler vector bundle and proved that a holomorphic vector bundle is Finsler flat in his sense if and only if it is Hermitian flat. In our opinion, Aikou's definition is too strong for use. In Section 4, we will give a new definition of a flat Finsler vector bundle based on the Kobayashi curvature. Moreover, by making use of a result of B. Berndtsson in \cite{Ber} and the first Chern form $\mathcal{C}_1(E,G)$, we prove that a holomorphic vector bundle is Finsler flat in our sense if and only if it admits a flat Hermitian metriic.
$$ \ $$
\noindent{\bf Acknowledgement:} The first author would like to thank Professor Weiping Zhang in Chern Institute of Math., Nankai Univ. for the encouragement on his study of Finsler geometry.

\section{Notations and some basic facts} \label{s1}

In this section, to make this paper easy to read, we fix some notations and recall some basic definitions and facts used in this paper. For more detail we refer to references \cite{AP}, \cite{Aikou}, \cite{Cao-Wong}, \cite{Gr-Har}, \cite{Ko1} and \cite{Ko2}.

Let $z=(z^1,\cdots,z^n)$ be a local coordinate system in $M$, and let $v=(v^1,\cdots,v^r)$ be the fibre coordinate system defined by a local holomorphic frame $s=\{s_1,\cdots,s_r\}$ of $E$. We write
\begin{align}\label{1.1111}
& G_{i}=\partial G/\partial v^{i},\quad G_{\bar j}=\partial G/\partial\bar{v}^{j},\quad G_{i\bar{j}}=\partial^{2}G/\partial v^{i}\partial\bar{v}^{j},\\
& G_{i\alpha}=\partial^{2}G/\partial v^{i}\partial z^{\alpha}, \quad G_{i\bar j\bar\beta}=\partial^3G/\partial v^{i}\partial\bar v^j\partial\bar z^{\beta},\quad etc.,
\end{align}
to denote the differentiation with respect to $v^i,\bar v^j$ ($1\leq i,j\leq r$), $z^\alpha,\bar z^\beta$ ($1\leq\alpha,\beta\leq n$). In the following lemma we collect some useful identities related to a Finsler metric $G$, in which we adopt the summation convention of Einstein. For a proof we refer to Kobayashi \cite{Ko2}.
\begin{lemma}\label{l.1}(cf. \cite{Ko2}) The following identities hold for any $(z,v)\in E^o$, $\lambda\in \mathbb{C}$:
\begin{align}\label{1.2}
G_i(z,\lambda v)=\bar\lambda G_i(z,v),\quad G_{i\bar j}(z,\lambda v)=G_{i\bar j}(z,v)=\bar G_{j\bar i}(z,v);
\end{align}
\begin{align}\label{1.3}
G_i(z,v)v^i=G_{\bar j}(z,v)\bar v^j=G_{i\bar j}(z,v)v^i\bar v^j=G(z,v);
\end{align}
\begin{align}\label{1.4}
G_{ij}(z,v)v^i=G_{i\bar j k}(z,v)v^i=G_{i\bar j\bar k}(z,v)\bar v^j=0.
\end{align}
\end{lemma}
Now when restricted to the fiber $E^o_z=\pi^{-1}(z)$ of $\pi:E^o\to M$, the Levi form can be written as
\begin{align}\label{1.5}
i^*_z\left({\sqrt{-1}}\partial\bar\partial G\right)={\sqrt{-1}}G_{i\bar j}(z,v)dv^id\bar v^j,
\end{align}
where $i_z:E_z^o\hookrightarrow E^o$ is the natural inclusion. Clearly, by the condition {\bf F4)} the Hermitian matrices $\left(G_{i\bar j}(z,v)\right)$ is positive definite and actually defines an Hermitian metric $h^G$ on the pull-back bundle $p:\pi^*E\to E^o$. From Lemma \ref{l.1} (\ref{1.2}),
the data $\left(G_{i\bar j}(z,v)\right)$ can be viewed naturally as the data defined on $P(E)$. Hence $h^G$ is actually an Hermitian metric
on the pull-back bundle $p:\pi^*E\to P(E)$. Also, as a subbundle of $\pi^*E$, the tautological line bundle
\begin{align}\label{1.6}
\mathcal{O}_{P(E)}(-1)=\left\{((z,[v]),Z)|Z=\lambda v, \lambda\in\mathbb{C}\right\}
\end{align}
inherits an Hermitian metric from $h^G$ on $\pi^*E$, which is still denoted by $h^G$. One verifies easily that $h^G(Z,Z)=G(Z)$ for any $Z\in \mathcal{O}_{P(E)}(-1)$.
Let $\nabla^{\pi^*E}$ denote the Chern connection on the holomorphic vector bundle $p:(\pi^*E,h^G)\to E^o$. With respect to the local holomorphic frame $s=\{s_1,\cdots,s_r\}$ of $\pi^*E$, where we abuse the notation $s_k$ to denote the sections of $\pi^*E$, then the Chern connection $(1,0)$-forms $\theta^k_i$ are given by
\begin{align}\label{1.5111}
\nabla^{\pi^*E}s_i=\theta^k_i\otimes s_k,\quad \theta^k_i=(\partial G_{i\bar j})G^{\bar j k}=\Gamma^k_{i\alpha}dz^\alpha+\gamma^k_{il}dv^l,
\end{align}
\begin{align}\label{1.5112}
\Gamma^k_{i\alpha}={{\partial G_{i\bar j}}\over{\partial z^\alpha}}G^{\bar j k},\quad \gamma^k_{il}={{\partial G_{i\bar j}}\over{\partial v^l}}G^{\bar j k},
\end{align}
where $(G^{{\bar j} k})$ denote the inverse of the matrix $(G_{i\bar j})$. From Lemma \ref{l.1}, one sees that
\begin{align}\label{1.5113}
\Gamma^k_{i\alpha}(z,\lambda v)=\Gamma^k_{i\alpha}(z,v),\quad \gamma^k_{il}(z,v)v^i=\gamma^k_{il}(z,v)v^l=0.
\end{align}
When $G$ is induced from an Hermitian metric on $E$, the data $\gamma^k_{il}$ vanish and $\Gamma^k_{i\alpha}$ descend to the base manifold $M$.

By using the Chern connection $\nabla^{\pi^*E}$ one gets a smooth h-v decomposition of the holomorphic tangent vector bundle $TE^o$ of $E^o$ (cf. \cite{Cao-Wong}, Sect.5):
\begin{align}\label{1.6111}
TE^o=\mathcal{H}\oplus\mathcal{V},
\end{align}
where $\mathcal{V}$ is called the vertical subbundle of $TE^o$ defined by
\begin{align}\label{1.6112}
\mathcal{V}=\ker(p_*:TE^o\to TM),
\end{align}
and $\mathcal{H}$ is called the horizontal subbundle of of $TE^o$ defined by
\begin{align}\label{1.6113}
\mathcal{H}=\ker(\nabla^{\pi^*E}_\bullet P:TE^o\to\pi^*E),
\end{align}
where $P$ is the tautological section of the bundle $p:\pi^*E\to E^0$ defined by $P(z,v)=v.$ Clearly, the vertical subbundle $\mathcal{V}$ is holomorphically isomorphic to $\pi^*E$ canonically, while the horizontal subbundle $\mathcal{H}$ is isomorphic to $\pi^*TM\to E^o$ smoothly. In local coordinates, one has
\begin{align}\label{1.6114}
\mathcal{H}={\rm span}_\mathbb{C}\left\{{\delta\over{\delta z^\alpha}}={\partial\over{\partial z^\alpha}}-\Gamma^k_{j\alpha}v^j{\partial\over{\partial v^k}},\ 1\leq\alpha\leq n\right\},\quad \mathcal{V}={\rm span}_\mathbb{C}\left\{{\partial\over{\partial v^i}},\ 1\leq i\leq r\right\}.
\end{align}
Moreover, the dual bundle $T^*E^o$ also has a smooth h-v decomposition $T^*E^o=\mathcal{H}^*\oplus\mathcal{V}^*$ with
\begin{align}\label{1.6115}
\mathcal{H}^*={\rm span}_\mathbb{C}\left\{dz^\alpha,\ 1\leq\alpha\leq n\right\},\quad \mathcal{V}^*={\rm span}_\mathbb{C}\left\{\delta v^k=dv^k+\Gamma^k_{j\alpha}v^jdz^\alpha,\ 1\leq i\leq r\right\}.
\end{align}
From (\ref{1.5113}), the natural action of the group $\mathbb{C}^*$ on $E^o$ preserves the horizontal subbundle $\mathcal{H}$. As a result, the quotient map $q:E^o\to P(E)=E^o/{\mathbb{C}^*}$ induces a smooth h-v decompositions of $TP(E)$ and $T^*P(E)$
\begin{align}\label{1.61111}
TP(E)=\tilde{\mathcal{H}}\oplus\tilde{\mathcal{V}},\ T^*P(E)={\tilde{\mathcal{H}}}^*\oplus{\tilde{\mathcal{V}}}^*,
\end{align}
where
\begin{align}\label{1.88888}
\tilde{\mathcal{H}}=q_*\mathcal{H},\quad\tilde{\mathcal{V}}=q_*\mathcal{V}.
\end{align}
By using the h-v decompositions above, one sees that
\begin{align}\label{1.6119}
\omega_{FS}={{\sqrt{-1}}\over{2\pi}}{{\partial^2\log G}\over{\partial v^i\partial\bar v^j}}\delta v^i\wedge\delta\bar v^j
\end{align}
is a well-defined $(1,1)$-form on $E^o$. Note that $\omega_{FS}$ is actually a vertical $(1,1)$-form on $P(E)$ and when restricted to each fibre of the fibration $\pi:P(E)\to M$, gives a K\"{a}hler metric on the fibre.
Let $R^{\pi^*E}$ denote the curvature of $\nabla^{\pi^*E}$. With respect to the local holomorphic frame $s=\{s_1,\cdots,s_r\}$, the Chern curvature $(1,1)$-forms are given by
\begin{align}\label{1.6116}
R^{\pi^*E}s_i=\Theta^k_i\otimes s_k,\quad \Theta^k_i=d\theta^k_i-\theta^l_i\wedge\theta^k_i=\bar\partial\theta^k_l.
\end{align}
Set
\begin{align}\label{1.7}
\Psi={\sqrt{-1}}{{h^G(R^{\pi^*E}P,P)}\over{h^G(P,P)}}.
\end{align}
Clearly, $\Psi$ is a well-defined real $(1,1)$-form on $E^o$. Moreover, in local coordinate systems $(z^\alpha,v^i)$, $\Psi$ can be written as (cf. \cite{Ko1}, \cite{Ko2}, page 149)
\begin{align}\label{1.8}
\Psi={\sqrt{-1}}K_{i\bar j \alpha\bar\beta}{{v^i\bar v^j}\over G}dz^\alpha\wedge d\bar z^\beta,\quad
K_{i\bar j \alpha\bar\beta}=-G_{i\bar j \alpha\bar\beta}+G^{k\bar l}G_{i\bar l\alpha}G_{k\bar j\bar\beta},
\end{align}
From (\ref{1.8}) one sees easily that $\Psi$ is actually a horizontal $(1,1)$-form defined on $P(E)$. We note that when $G$ comes from a Hermitian metric $h$ on $E$, then the form $\Psi$ descends to $M$ and is just the holomorphic bisectional curvature of $h$.

\begin{defn}\label{d1.1}
 The form $\Psi$ defined by (\ref{1.7}) is called the Kobayashi curvature of the holomorphic Finsler vector bundle $(E,G)$. A strongly pseudo-convex complex Finsler metric $G$ is said to be of positive (resp. negative) Kobayashi curvature if for any nonzero horizontal vectors $X\in\mathcal{H}$, one has
\begin{align}\label{1.9}
-{\sqrt{-1}}\Psi(X,\bar X)>0\ ({\mbox{\rm resp.}}\ <0),
\end{align}
or equivalently, the matrix $(K_{i\bar j \alpha\bar\beta}v^i\bar v^j)$
is positive (resp. negative)-definite on $E^o$.
\end{defn}

Now consider the hyperplane line bundle $(\mathcal{O}_{P(E)}(1)$, which is the dual of $\mathcal{O}_{P(E)}(-1)$ with the dual metric $(h^G)^{-1}$. Then $\partial\bar\partial\log G$ gives the curvature of $(h^G)^{-1}$ and
\begin{align}\label{1.10}
\Xi:=c_1(\mathcal{O}_{P(E)}(1),(h^G)^{-1})={{\sqrt{-1}}\over{2\pi}}\partial\bar\partial\log G.
\end{align}
is the first Chern form associated to the metric $(h^G)^{-1}$.

For readers' convenience, we give a proof of the following very important lemma due to Kobayashi and Aikou (cf. \cite{Ko1}, \cite{Ko2}, \cite{Aikou}).
\begin{lemma}\label{l.2}(Kobayashi, Aikou) Let $\pi:E\to M$ be a holomorphic vector bundle with a strongly pseudo-convex complex Finsler metric $G$. Then
\begin{align}\label{1.11}
\Xi=c_1(\mathcal{O}_{P(E)}(1),(h^G)^{-1})=-{1\over{2\pi}}\Psi+\omega_{FS}.
\end{align}
\end{lemma}
\begin{proof}
We do computations by using the definitions (\ref{1.5112}) and (\ref{1.6115}):
\begin{align*}
\frac{\partial^{2}\log G}{\partial v^{j}\partial\bar{v}^{k}}\delta v^{j}\wedge\delta\bar{v}^{k}&=\frac{\partial^{2}\log G}{\partial v^{j}\partial\bar{v}^{k}}d v^{j}\wedge d\bar{v}^{k}+\frac{\partial^{2}\log G}{\partial v^{j}\partial\bar{v}^{k}}\Gamma^{j}_{l\alpha}v^{l}dz^{\alpha}\wedge d\bar{v}^{k}\\
&+\frac{\partial^{2}\log G}{\partial v^{j}\partial\bar{v}^{k}}\overline{\Gamma^{k}_{s\beta}v^{s}}d v^{j}\wedge d\bar{z}^{\beta}
+\frac{\partial^{2}\log G}{\partial v^{j}\partial\bar{v}^{k}}\Gamma^{j}_{l\alpha}v^{l}\overline{\Gamma^{k}_{s\beta}v^{s}}dz^{\alpha}\wedge d\bar{z}^{\beta},
\end{align*}
and
\begin{align*}
\frac{\partial^{2}\log G}{\partial v^{j}\partial\bar{v}^{k}}\Gamma^{j}_{l\alpha}v^{l}=\frac{GG_{j\bar{k}}-G_{j}G_{\bar{k}}}{G^{2}}G_{\bar{h}\alpha}G^{\bar{h}j}=\frac{\partial^{2}\log G}{\partial z^{\alpha}\partial\bar{v}^{k}},
\end{align*}
$$\frac{\partial^{2}\log G}{\partial v^{j}\partial\bar{v}^{k}}\overline{\Gamma^{k}_{s\beta}v^{s}}=\overline{\frac{\partial^{2}\log G}{\partial v^{k}\partial\bar{v}^{j}}\Gamma^{k}_{s\beta}v^{s}}=\frac{\partial^{2}\log G}{\partial \bar{z}^{\beta}\partial v^{j}},$$
\begin{align*}
\frac{\partial^{2}\log G}{\partial v^{j}\partial\bar{v}^{k}}\Gamma^{j}_{l\alpha}v^{l}\overline{\Gamma^{k}_{s\beta}v^{s}}=\frac{1}{G^{2}}(GG_{\bar{k}\alpha}G_{h\bar{\beta}}
G^{\bar{k}h}-G_{\alpha}G_{\beta}).
\end{align*}
Thus by (\ref{1.8}), one gets
\begin{align*}
&\partial\bar{\partial}\log G-\frac{\partial^{2}\log G}{\partial v^{j}\partial\bar{v}^{k}}\delta v^{j}\wedge\delta\bar{v}^{k}\\
&=[\frac{1}{G^{2}}(GG_{\alpha\bar{\beta}}-G_{\alpha}G_{\beta})-\frac{1}{G^{2}}(GG_{\bar{k}\alpha}G_{h\bar{\beta}}G^{\bar{k}h}
-G_{\alpha}G_{\beta})]dz^{\alpha}
\wedge d\bar{z}^{\beta}\\
&=\frac{1}{G}(G_{\alpha\bar{\beta}}-G_{\bar{k}\alpha}G_{h\bar{\beta}}G^{\bar{k}h})dz^{\alpha}
\wedge d\bar{z}^{\beta}\\
&=-K_{i\bar{j}\alpha\bar{\beta}}\frac{v^{i}\bar{v}^{j}}{G}dz^{\alpha}\wedge d\bar{z}^{\beta}={\sqrt{-1}}\Psi.
\end{align*}
\end{proof}

When the strongly pseudo-convex Finsler metric $G$ is of negative Kobayashi curvature, then by (\ref{1.11}) the first Chern form $\Xi=c_1(\mathcal{O}_{P(E)}(1),(h^G)^{-1})$ is a closed positive $(1,1)$-from on $P(E)$, which gives a K\"{a}hler form on $P(E)$. So in this case $P(E)$ is a K\"{a}helr manifold with the K\"{a}hler form $\Xi$.

We recall that a holomorphic vector bundle $E\to M$ over a closed complex manifold $M$ is ample in the sense of Hartshorne (cf. \cite{Hart}) if and only if
there is an Hermitian metric $h$ on $\mathcal{O}_{P(E^*)}(1)$ such that the first Chern form $c_1(\mathcal{O}_{P(E^*)}(1),h)$ is  positive. From this equivalent characterization of the ampleness of $E$ and Lemma \ref{l.2}, one has the following Kobayashi's characterization of the ampleness (cf. \cite{Ko1}):
\begin{thm}\label{t.1}(Kobayashi,1975) A holomorphic vector bundle $E\to M$ over a closed complex manifold $M$ is ample in the sense of Hartshorn if and only if there exists a Finsler metric $G$ on the dual bundle $E^{*}$ of $E$ with  negative Kobayashi curvature.
\end{thm}

\section{Segre forms and Chern forms from Finsler metrics} \label{s2}

In this section, we give two kinds of Chern forms $c_k(E,G)$ and $\mathcal{C}_k(E,G)$ as well as Segre forms $s_k(E,G)$ of a holomorphic Finsler vector bundle $(E,G)$, which are closed $(k,k)$-forms on $M$ expressed by the strongly pseudo-convex complex Finsler metric $G$ and represent the Chern classes $c_k(E)$ and Segre classes $s_j(E)$ of $E$, respectively, where $0\leq k\leq r$, $0\leq j\leq n$, and work out a Bott-Chern transgression term of $\mathcal{C}(E,G)$ and $c(E,G)$. So to some extent, we answer Faran's question mentioned in the introduction of this paper. As an application, we will show that the signed Segre forms $s_k(E,G)$ are positive $(k,k)$ forms on $M$ for the positive Finsler metric $G$.

As in the introduction of this paper, we denote $c_1(\mathcal{O}_{P(E)}(1),(h^G)^{-1})$ by $\Xi$, which is a closed $(1,1)$-form on $P(E)$.
From the following well-known isomorphism (cf. \cite{BT}, \S21)
\begin{align}\label{2.1}
T(P(E)/M)\cong\left({p^*E/{\mathcal{O}_{P(E)}(-1)}}\right)\otimes\mathcal{O}_{P(E)}(1)),
\end{align}
where $T(P(E)/M)$ denote the vertical tangent bundle (or relative tangent bundle) of the fibration $\pi:P(E)\to M$, one gets for any $z\in M$ that
\begin{align}\label{2.2}
ri^*_zc_1(\mathcal{O}_{P(E)}(1))^{r-1}=c_{r-1}(P(E_z))=e(P(E_z)),
\end{align}
where $i_z:P(E_z)\hookrightarrow P(E)$ is the natural inclusion. As a result, one has
\begin{lemma}\label{l.2.1} $\Xi^{r-1}$ is a closed $(r-1,r-1)$-form on $P(E)$, and for any $z\in M$,
\begin{align}\label{2.3}
\int_{P(E_z)}i^*_z\Xi^{r-1}=1.
\end{align}
\end{lemma}
So $\Xi^{r-1}$ serves as a ``Thom form" on $P(E)$ and helps to push-down closed forms on $P(E)$ to ones on $M$.

By the standard Chern-Weil theory (cf. \cite{Zhang}), the total Chern form $c(\pi^*E,h^G)$ of $(\pi^*E,h^G)$ defined by
\begin{align}\label{2.4}
c(\pi^*E,h^G)=\det\left(I+\frac{\sqrt{-1}}{2\pi}R^{\pi^{*}E}\right)\in\bigoplus_{k=0}^r\mathcal{A}^{(k,k)}(P(E))
\end{align}
is a closed form on $P(E)$ and represents the total Chern class of $\pi^*E$. Let $c_k(\pi^*E,h^G)$ denote the $(k,k)$-part in $c(\pi^*E,h^E)$, which is the $k^{th}$ Chern forms of $(\pi^*E,h^G)$ and represents the $k^{th}$ Chern class of $E$ for $0\leq k\leq r$.

Now by using $\Xi^{r-1}$, we can define the following Chern forms $c_k(E,G)$ and total Chern form $c(E,G)$ by
\begin{align}\label{2.5}
c_k(E,G)=\int_{P(E)/M}c_k(\pi^*E,h^G)\Xi^{r-1},\quad c(E,G)=\sum^r_{k=0}c_k(E,G).
\end{align}
\begin{lemma}\label{l2.2} For each $k$, $0\leq k\leq r$, the Chern form $c_k(E,G)$ is a closed $(k,k)$-form on $M$ and represents the $k^{th}$ Chern class $c_k(E)$ of $E$.
\end{lemma}
\begin{proof} Giving any Hermitian metric $h$ on $E$, we compute explicitly the transgression term of the total Chern form $c(E,G)$ and
$\pi^*c(E,h)$. Note that
\begin{align}\label{2.6}
\pi^*c(E,h)=c(\pi^*E,\pi^*h)=\det\left(I+{{\sqrt{-1}}\over{2\pi}}\pi^*R^E\right),
\end{align}
where $R^E$ denote the curvature of the Chern connection $\nabla^E$ determined by the Hermitian metric $h$ on $E$.
Now consider the family of Hermitian metrics $h_t=t\pi^*h+(1-t)h^G$, $0\leq t\leq 1$ on $\pi^*E$. Let $\nabla_t$ be the Chern connection on $\pi^*E$ associated to the metric $h_t$ and let $R_t$ denote the curvature of $\nabla_t$. Then by \cite{Bo-C}, Proposition 3.28, one has
\begin{align}\label{2.7}
c(\pi^*E,h^G)-\pi^*c(E,h)=c(\pi^*E,h^G)-c(\pi^*E,\pi^*h)=\partial\bar{\partial}\xi,
\end{align}
where
\begin{align}\label{2.8}
\xi=\int_{0}^{1}\sum_{i=1}^{r}{\rm Det}(I+\frac{\sqrt{-1}}{2\pi}R_t,\cdots,\frac{\sqrt{-1}}{2\pi}\dot{L}_t,\cdots,I+\frac{\sqrt{-1}}{2\pi}R_t)dt.
\end{align}
where ${\rm Det}$, the polarization of $\det$, is the symmetric $r$-linear form on $\Omega(P(E),{\rm End}(\pi^*E))$, whose restriction on the diagonal is $\det$; and $\dot{L}_t=({dh_t}/dt)h^{-1}_t$. Set
\begin{align}\label{2.9}
\tilde\xi=\int_{P(E)/M}\xi\wedge\Xi^{r-1}.
\end{align}
One has from (\ref{2.5}), (\ref{2.7}) and Lemma \ref{l.2.1},
\begin{eqnarray*}
c(E,G)-c(E,h)&=&\int_{P(E)/M}c(\pi^*E,h^G)\Xi^{r-1}-\int_{P(E)/M}\pi^*c(E,h)\Xi^{r-1}\\
&=&\int_{P(E)/M}\left(c(\pi^*E,h^G)-c(\pi^*E,\pi^*h)\right)\Xi^{r-1}\\
&=&\int_{P(E)/M}(\partial\bar{\partial}\xi)\Xi^{r-1}=\int_{P(E)/M}\partial\bar{\partial}(\xi\Xi^{r-1})\\
&=&\partial\bar{\partial}\int_{P(E)/M}\xi\Xi^{r-1}=\partial\bar{\partial}\tilde\xi.
\end{eqnarray*}
Now from the equality
\begin{align}\label{2.10}
c(E,G)=c(E,h)+\partial\bar{\partial}\tilde\xi\in\bigoplus_{k=0}^r\mathcal{A}^{(k,k)}(M),
\end{align}
one sees easily that total Chern form $c(E,G)$ and the Chern forms $c_k(E,G)$ are closed froms on $M$ and represent the total Chern class $c(E)$ and $c_k(E)$, respectively.
\end{proof}

In the following, we will give another kind of Chern forms expressed by Finsler metrics through the so-called Segre forms. Topologically, Segre classes
$s_j(E)$ of $E$, as the cohomology classes in $H^*(M,\mathbb{Z})$, are defined as the direct image $\pi_*c_1(\mathcal{O}_{P(E)}(1))^{r-1+j}$ of the powers of the first Chern class of $\mathcal{O}_{P(E)}(1)$. The total Segre class $s(E)$ of $E$ is defined by
\begin{align}\label{2.11}
s(E)=1+s_1(E)+\cdots+s_n(E).
\end{align}
A well-known relation between Chern classes and Segre classes of $E$ is
\begin{align}\label{2.12}
c(E)=s(E)^{-1},
\end{align}
which can also be viewed as a definition of Chern classes. Motivated by this, people define the following Segre forms $s_j(E,h)\in\mathcal{A}^{(j,j)}(M)$ of $E$ by using an Hermitian metric $h$ on $E$ (cf. \cite{Diverio}, \cite{Guler}, \cite{Mour}):
\begin{align}\label{2.13}
s_j(E,h)=\int_{P(E)/M}c_1(\mathcal{O}_{P(E)}(1),h)^{r-1+j},\quad s(E,h)=\sum^n_{j=0}s_j(E,h).
\end{align}
Moreover, at the differential form level, it holds that (cf. \cite{Diverio}, \cite{Guler}, \cite{Mour}),
\begin{align}\label{2.14}
c(E,h)=s(E,h)^{-1}.
\end{align}
For simplicity, we also denote the integral $\int_{P(E)/M}\phi$ of a differential form $\phi$ on $P(E)$ along fibres simply by $\pi_*\phi$.

In Finsler case, we define the following Segre forms $s_j(E,G)$ and another total Chern form $\mathcal{C}(E,G)$ of $(E,G)$:
\begin{align}\label{2.15}
s_j(E,G):=\pi_*\Xi^{r-1+j},\quad s(E,G)=\sum^n_{j=0}s_j(E,G),\quad \mathcal{C}(E,G)=s(E,G)^{-1}.
\end{align}
One see easily
\begin{align}\label{2.16}
\begin{split}
&\mathcal{C}_1(E,G)=-s_1(E,G),\quad \mathcal{C}_2(E,G)=s_1(E,G)^2-s_2(E,G),\\
&\mathcal{C}_3(E,G)=-s_3(E,G)^3+2s_1(E,G)s_2(E,G)-s_3(E,G),\ etc.
\end{split}
\end{align}

\begin{lemma}\label{l2.3} For each $k$, $0\leq k\leq r$, the Chern form $\mathcal{C}_k(E,G)$ is a closed $(k,k)$-form on $M$ and represents the $k^{th}$ Chern form $c_k(E)$ of $E$.
\end{lemma}
\begin{proof} Giving any Hermitian metric $h$ on $E$, we compute explicitly the transgression term of the Chern form $\mathcal{C}(E,G)$ and
$\pi^*c(E,h)$. Set
\begin{align}\label{2.17}
\varphi=\frac{\sqrt{-1}}{2\pi}\log{{h^G}\over h}.
\end{align}
Clearly, $\varphi$ is a smooth function on $P(E)$. An easy computation shows
\begin{align}\label{2.18}
\Xi-c_1(\mathcal{O}_{P(E)}(1),h^{-1})=\partial\bar{\partial}\varphi.
\end{align}
So
\begin{align}\label{2.19}
s_j(E,G)=&\pi_*\Xi^{r-1+j}=\pi_*\left(c_1(\mathcal{O}_{P(E)}(1),h^{-1})+\partial\bar{\partial}\varphi\right)^{r-1+j}
=s_j(E,h)+\partial\bar{\partial}\eta_j,
\end{align}
where
\begin{align}\label{2.20}
\eta_j=\pi_*\left(\varphi \sum_{i=1}^{r-1+j}\binom{r-1+j}{i}c_1(\mathcal{O}_{P(E)}(1),h^{-1})^{r-1+j-i}(\partial\bar{\partial}\varphi)^{i-1}\right).
\end{align}
Therefore, the total Chern form $\mathcal{C}(E,G)$ is
\begin{align*}
\mathcal{C}(E,G)&=\left(\sum_{j=0}^ns_j(E,G)\right)^{-1}=\left(1+\sum_{j=1}^n(s_j(E,h)+\partial\bar{\partial}\eta_j)\right)^{-1}\\
&=1+\sum_{i=1}^n(-1)^i\left(\sum_{j=1}^ns_j(E,h)+\partial\bar{\partial}\eta\right)^i\\
&=1+\sum_{i=1}^n(-1)^i\left(\sum_{j=1}^n(s_j(E,h)\right)^i+\partial\bar{\partial}\tilde\eta\\
&=s(E,h)^{-1}+\partial\bar{\partial}\tilde\eta=c(E,h)+\partial\bar{\partial}\tilde\eta\\
&=c(E,G)+\partial\bar{\partial}(\tilde\eta-\tilde\xi),
\end{align*}
where the last two equalities come from (\ref{2.14}) and (\ref{2.10}), and
\begin{align}\label{2.21}
&\eta=\sum_{i=1}^n\eta_j,\
\tilde\eta=\eta\sum^n_{i=1}\sum_{l=1}^i(-1)^i\binom{i}{l}\left(\sum_{j=1}^ns_j(E,h)\right)^{i-l}(\partial\bar{\partial}\eta)^{l-1}.
\end{align}
Now the lemma follows from Lemma \ref{l2.2} and the equality $c(E,G)=\mathcal{C}(E,G)+\partial\bar{\partial}(\tilde\eta-\tilde\xi)$.
\end{proof}
As a consequence, by a well-known fact in complex geometry (cf. \cite{Bo-C}), we can express the Euler characteristic $\chi(M)$ of $M$ by these two kinds of Chern forms.
\begin{cor} Let $M$ be a closed complex manifold of dimension $n$. Let $G$ be a strongly pseudo-convex complex Finsler metric on the holomorphic tangent vector bundle $TM$ of $M$. Then
$$\chi(M)=\int_Mc_n(TM,G)=\int_M\mathcal{C}_n(TM, G).$$
\end{cor}
\begin{rem} In some sense, the above corollary can be viewed as a Gauss-Bonnet-Chern formula in the complex Finsler geometry setting. On the other hand, we can do all things in this section similarly for complex Finsler vector bundles over a closed smooth manifold. As a result, we can also deduce a Gauss-Bonnet-Chern-type formula for an almost complex manifold with a strongly pseudo-convex complex Finsler metric on the tangent vector bundle $TM$ of $M$, here $TM$ is viewed as a complex vector bundle. The defect in it is that the metric-preserving connection is not unique as in the holomorphic case.
\end{rem}
We denote the form $\tilde\xi-\tilde\eta$ by ${\tilde c}(E,G;h)$, which is a Bott-Chern transgression term of $c(E,G)$ and $\mathcal{C}(E,G)$.
Let ${\tilde c}_k(E,G;h)$ denote the $(k,k)$-part in ${\tilde c}(E,G;h)$. Clearly, one has
\begin{align}\label{2.c1}
c_k(E,G)-\mathcal{C}_k(E,G)=\partial\bar{\partial}{\tilde c}_{k-1}(E,G;h),\quad {\tilde c}(E,G;h)=\sum^{r-1}_{k=0}{\tilde c}_k(E,G;h).
\end{align}
\begin{ex} As an example, we write down the following Bott-Chern transgression term ${\tilde c}_0(E,G;h)$ by using formulas (\ref{2.8}), (\ref{2.9}) and ({\ref{2.21}}):
 \begin{align}\label{2.22}
 \begin{split}
 {\tilde c}_0(E,G;h)&=\frac{\sqrt{-1}}{2\pi}\int_{P(E)/M}
\left(\log\frac{G}{h}\sum_{i=0}^{r-1}(\frac{\sqrt{-1}}{2\pi}\partial\bar{\partial}\log G)^{i}(\frac{\sqrt{-1}}{2\pi}\partial\bar{\partial}\log h)^{r-1-i}\right.\\
& \left.-\log\frac{\det(G_{i\bar{j}})}{\det(h_{i\bar{j}})}(\frac{\sqrt{-1}}{2\pi}\partial\bar{\partial}\log G)^{r-1}\right).
 \end{split}
 \end{align}
\end{ex}

\begin{rem}\label{r2.1}
Note that $\tilde\xi$, $\tilde\eta$ depend on the choice of the Hermitian metric $h$ on $E$ and it looks that the Bott-Chern form ${\tilde c}(E,G;h)$ also depends on $h$. We do not know whether it holds that $c(E,G)=\mathcal{C}(E,G)$ as the Hermitian case (\ref{2.14}). However, by averaging the induced Hermitian metric $h^G$ alone the fibres of $P(E)$, one gets a Hermitian metric $h(G)$ on $E$; then by using this metric, one gets a Bott-Chern transgression term ${\tilde c}(E,G;h(G))$, which only depends on $G$ itself.
\end{rem}

In the following, we will prove the positivity of the signed Segre forms $(-1)^ks_k(E,G)$ for a Finsler metric $G$ on $E$ with positive Kobayashi
curvature.

We recall firstly that a smooth $(p,p)$-form $\phi$ on a complex manifold $N$ is positive if for any $x\in N$ and any linearly independent $(1,0)$ type tangent vectors $v_1,v_2,\cdots, v_p$ at $x$, it holds that
\begin{align}\label{2.111}
(-\sqrt{-1})^{p^2}\phi(v_1,v_2,\cdots, v_p,\bar v_1,\bar v_2,\cdots, \bar v_p)>0.
\end{align}

\begin{thm}\label{t6.2} If $E$ admits a Finsler metric $G$ of the positive (resp. negative) Kobayashi curvature, then the signed Segre form $(-1)^{k}s_{k}(E,G)$ (resp. the Segre forms $s_{k}(E,G)$) are positive $(k,k)$-forms for $0\leq k\leq n$.
\end{thm}
\begin{proof} We only prove the lemma for the positive Kobayashi curvature case, the proof of the other case is the same. In this case, for any $(z_0,[v_0]\in P(E)$, there exists a basis $\{\psi^1,\cdots,\psi^n\}$ for ${\tilde{\mathcal{H}}}^*|_{(z_0,[v_0])}$ such that
\begin{align}\label{2.112}
\Psi={\sqrt{-1}}\sum^n_{\alpha=1}\psi^\alpha\wedge\bar\psi^\alpha,
\end{align}
and so for $k\geq 1$,
\begin{align}\label{2.113}
(-{\sqrt{-1}})^{k^2}\Psi^k=k!\sum_{1\leq\alpha_1<\cdots<\alpha_k\leq n}\psi^{\alpha_1}\wedge\cdots\wedge\psi^{\alpha_k}\wedge\bar\psi^{\alpha_1}\wedge\cdots\wedge\bar\psi^{\alpha_k}.
\end{align}
Hence for any independent vectors $X_1,\cdots,X_k\in(q_*\mathcal{H})^*|_{(z_0,[v_0])}$, one has
\begin{align}\label{2.114}
\begin{split}
&(-{\sqrt{-1}})^{k^2}\Psi^k(X_1,\cdots,X_k,\bar X_1,\cdots,X_k)\\
&=k!\sum_{1\leq\alpha_1<\cdots<\alpha_k\leq n}\psi^{\alpha_1}\wedge\cdots\wedge\psi^{\alpha_k}\wedge\bar\psi^{\alpha_1}\wedge\cdots\wedge\bar\psi^{\alpha_k}
(X_1,\cdots,X_k,\bar X_1,\cdots,X_k)\\
&=k!\sum_{1\leq\alpha_1<\cdots<\alpha_k\leq n}|\psi^{\alpha_1}\wedge\cdots\wedge\psi^{\alpha_k}(X_1,\cdots,X_k)|^2> 0.
\end{split}
\end{align}
On the other hand, from (\ref{2.15}) and (\ref{1.11}), one gets
\begin{align}\label{2.115}
(-1)^{k}s_{k}(E,G)=\frac{1}{(2\pi)^{k}}\binom{r-1+k}{k}\int_{P(E)/M}\Psi^{k}\omega^{r-1}_{FS}.
\end{align}
For any $z_0\in M$ and any linearly independent $(1,0)$-type tangent vectors $Y_1,\cdots,Y_{k}$ in $T_{z_0}M$, one has
\begin{align*}
&(-\sqrt{-1})^{k^2}(-1)^{k}s_{k}(E,G)(Y_1,\cdots,Y_{k},\bar Y_1,\cdots,\bar Y_k)\\
&=\frac{1}{(2\pi)^{k}}\binom{r-1+k}{k}\int_{P(E_{z_0})}
(-\sqrt{-1})^{k^2}\Psi^{k}(Y^h_1,\cdots,Y^h_{k},\bar Y^h_1,\cdots,\bar Y^h_{k})i_{z_0}^*\omega^{r-1}_{FS},
\end{align*}
where $Y^h$ denote the horizontal lifting along the fibre $P(E_{z_0})$ of a vector $Y\in T_{z_0}M$, and $i_{z_0}:P(E_{z_0})\hookrightarrow P(E)$
is the inclusion map. Now the theorem follows from that the smooth function
$$(-\sqrt{-1})^{k^2}\Psi^k(Y^h_1,\cdots,Y^h_k,\bar Y^h_1,\cdots,\bar Y^h_k)$$
on $P(E_{z_0})$ is positive everywhere.
\end{proof}

\begin{cor}\label{6.3}
If $E$ is ample, then the signed Segre classes $(-1)^{k}s_{k}(E)$ of $E$ can be represented by a positive $(k,k)$ form. In particular, $c_1(E)$
can be represented by a positive $(1,1)$ form.
\end{cor}
\begin{proof} If $E$ ample, then $E^*$ admits a Finsler metric $G$ of the negative Kobayashi curvature. By Theorem \ref{t6.2}, $s_{k}(E^{*})$ can be represented by the positive $(k,k)$-form $s_k(E^*,G)$. Since $s_k(E)=(-1)^{k}s_k(E^{*})$, the signed Segre classes $(-1)^{k}s_{k}(E)$ can be represented by the positive $(k,k)$-form $s_k(E^*,G)$. Note that $c_1(E)=-s_1(E)$, so it can be represented by the positive $(1,1)$-form on $M$.
\end{proof}

\section{The semi-stability of a Finsler-Einstein vector bundle}

In this section, we discuss some properties of Finsler-Einsler vector bundles in the sense of Kobayashi.  By using
the first Chern form $\mathcal{C}_1(E,G)$ given in the last section, we will show, under an extra assumption, that a holomorphic vector bundle with a Finsler-Einsler metric $G$ is semi-stable. Moreover, we prove a Kobayashi-L\"{u}bke type inequality related to the Chern forms $\mathcal{C}_1(E,G)$ and $\mathcal{C}_2(E,G)$ of a Finsler-Einstein vector bundle $(E,G)$, which generalizes a recent result of Diverio (\cite{Diverio} in Hermitian case.

We first recall Kobayashi's definition of a Finsler-Einstein metric.
\begin{defn}\label{4.19}(cf. \cite{Ko2}) Let $(M,\omega)$ be a Hermitian manifold with the K\"{a}hler form $\omega={\sqrt{-1}}g_{\alpha\bar\beta}dz^\alpha\wedge d\bar z^\beta$. Let $E$ be a holomorphic vector bundle over $M$. A strongly pseudo-convex Finsler metric $G$ on $E$ is said to be Finsler-Einstein if there exists some constant $\lambda$ such that
\begin{align}\label{4.10}
tr_{\omega}\Psi:=g^{\alpha\bar\beta}K_{i\bar j \alpha\bar\beta}{{v^i\bar v^j}\over G}=\lambda,
  \end{align}
where $\Psi$ is the Kobayashi curvature of $G$.
\end{defn}
Recall that locally, $\Psi$ is written as (cf. (\ref{1.8}))
$$\Psi= {\sqrt{-1}}K_{i\bar j \alpha\bar\beta}{{v^i\bar v^j}\over G}dz^\alpha\wedge d\bar z^\beta.$$
\begin{rem}
  If the Finsler metric $G$ comes from an Hermitian metric, then $K_{i\b{j}\alpha\b{\beta}}$ is independent of the fiber and (\ref{4.10}) gives
\begin{align}\label{4.a11}
g^{\alpha\b{\beta}}K_{i\b{j}\alpha\b{\beta}}v^i\b{v}^j=\lambda G.
\end{align}
Differentiating both sides of (\ref{4.a11}) with respect to $v^i$, we have
\begin{align}\label{4.a12}
g^{\alpha\b{\beta}}K_{i\b{j}\alpha\b{\beta}}=\lambda G_{\alpha\b{\beta}}.
\end{align}
So
\begin{align}\label{4.3113}
g^{\alpha\b{\beta}}K^i_{j\alpha\b{\beta}}=\lambda\delta^i_j,
\end{align}
i.e., $G$ is an Hermitian-Einstein metric.
\end{rem}
Now we assume that $M$ is a K\"{a}hler manifold with a closed K\"{a}hler form $\omega$. By Lemma \ref{l.2} and (\ref{2.15}), (\ref{2.16}),
one has
\begin{align}\label{4.11}
\mathcal{C}_{1}(E,G)=\frac{r}{2\pi}\int_{P(E)/M}\Psi\w\omega^{r-1}_{FS}.
\end{align}
Thus by (\ref{4.10}),
\begin{align}\label{4.b11}
\mathcal{C}_{1}(E,G)\w \omega^{n-1}=\frac{r}{2\pi n}\left(\int_{P(E)/M}tr_{\omega}(\Psi)\omega^{r-1}_{FS}\right)\w\omega^{n}=\frac{\lambda r}{2\pi n}\w\omega^{n}.
\end{align}
So
\begin{align}\label{4.b11}
\lambda=2\pi\frac{n}{\int_{M}\omega^{n}}\frac{\int_{M}\mathcal{C}_{1}(E,G)\w \omega^{n-1}}{r}.
\end{align}
So as in the Hermitian-Einstein case, we also have that $\lambda$ depends only on the K\"{a}hler class $[\omega]$ and the first Chern class $c_{1}(E)$ of $E$.

For a holomorphic vector bundle $E\to (M,\omega)$, the degree $\deg_{\omega}E$ and the slope $\mu(E)$ is defined respectively by
\begin{align}\label{4.b102}
\deg_{\omega}E=\langle c_1(E)\wedge[\omega^{n-1}],[M]\rangle,\quad \mu(E)={{\deg_{\omega}E}\over{\text{rank}(E)}}.
 \end{align}
Given a coherent sheaf $\mathcal{F}$ over $M$, one can get a holomorphic line bundle $\det(\mathcal{F})$ over $M$ in a standard way (cf. \cite{Ko3}, page 162). The first Chern class $c_1(\mathcal{F})$ of $\mathcal{F}$ is defined to be the first Chern class $c_1(\det(\mathcal{F}))$ of $\det(\mathcal{F})$. Then the degree $\deg_{\omega}\mathcal{F}$ and the slope $\mu(\mathcal{F})$ of $\mathcal{F}$ with respect to $\omega$ are defined respectively by
\begin{align}\label{4.b103}
\deg_{\omega}\mathcal{F}=\int_{M}c_{1}(\mathcal{F})\wedge \omega^{n-1},\quad \mu(\mathcal{F}):=\deg_{\omega}\mathcal{F}/\text{rank} \mathcal{F}.
\end{align}

We recall that $E$ is said to be $\omega$-\emph{stable} (resp. $\omega$-\emph{semi-stable}) if, for every coherent subsheaf $\mathcal{F}$ of $\mathcal{O}(E)$, $0<{\rm rank}(\mathcal{F})<{\rm rank}(E)$, one has
\begin{align}\label{4.104}
\mu(\mathcal{F})<\mu(E)\quad (resp.\quad \mu(\mathcal{F})\leq\mu(E)).
\end{align}

\begin{lemma}\label{p3}
  Let $(E,G)\to (M,\omega)$ be a Finsler-Einstein vector bundle over a closed K\"{a}hler manifold $M$ with the K\"{a}hler form $\omega$. For any holomorphic subbundle $E'\subset E$ of $E$ with $r'={\rm rank}(E')\leq r={\rm rank}(E)$, then
  \begin{align}
    \mu(E')\leq \mu(E).
  \end{align}
\end{lemma}

\begin{proof}
 For any holomorphic subbundle $E'\subset E$ of $E$ with $r'={\rm rank}(E')\leq r={\rm rank}(E)$ and consider the restricted Finsler metric $G|_{E'}$. Locally, for each point $p\in P(E)$, we can choose a local frame $\{s_1,\cdots,s_{r'},\cdots s_r\}$ of $E$ such that $\{s_1,\cdots,s_{r'\}}$ is a local frame for $E'$, and a normal coordinate system near $p\in P(E')\subset P(E)$ such that $G_{i\b{j}\alpha}=0$ at $p$ (cf. \cite{Ko1}, \cite{Ko2}).

We make the following convention for indices:
$$1\leq A,B,\cdots\leq r,\quad 1\leq i,j,k,l\leq r'<a,b,c,d\leq r.$$
Set
\begin{align}\label{4.2225}
T_{i\b{j}\alpha\b{\beta}}=K_{i\b{j}\alpha\b{\beta}}-K'_{i\b{j}\alpha\b{\beta}},
\end{align}
where $K_{i\b{j}\alpha\b{\beta}}$ and $K'_{i\b{j}\alpha\b{\beta}}$ are defined for $G$ and $G|_{E'}$ as in (\ref{1.8}), respectively.
Then
\begin{align}\label{4.2222}
T_{i\b{j}\alpha\b{\beta}}=G_{i\b{a}\alpha}G^{\b{a}b}G_{b\b{j}\b{\beta}}+G_{i\b{a}\alpha}G^{\b{a}k}G_{k\b{j}\b{\beta}}.
+G_{i\b{k}\alpha}G^{\b{k}b}G_{b\b{j}\b{\beta}}.
\end{align}
Moreover, one obtains at $p$,
\begin{align}\label{4.104}  g^{\alpha\b{\beta}}T_{i\b{j}\alpha\b{\beta}}\frac{v^{i}\b{v}^{j}}{G}
=\frac{1}{G}g^{\alpha\b{\beta}}G^{b\b{a}}(G_{i\b{a}\alpha}v^{i})(\o{G_{j\b{b}\beta}v^{j}})\geq 0.
\end{align}
For any $v=v^{i}s_{i}\in E'$, one has
  \begin{align*}
  \begin{split}
    &\frac{{\mathcal{C}_{1}(E',G|_{E'})\wedge\omega^{n-1}}}{r'}
    =\frac{\sqrt{-1}}{2\pi}\int_{P(E')/M}(K'_{i\b{j}\alpha\b{\beta}}\frac{v^{i}\b{v}^{j}}{G}dz^{\alpha}\wedge d\b{z}^{\beta})\omega^{r'-1}_{FS}\wedge\omega^{n-1}\\
    &=\frac{1}{2\pi n}\int_{P(E')/M}(g^{\alpha\b{\beta}}K'_{i\b{j}\alpha\b{\beta}}\frac{v^{i}\b{v}^{j}}{G})\omega^{r'-1}_{FS}\wedge\omega^{n}\\
    &=\frac{1}{2\pi n}\int_{P(E')/M}(g^{\alpha\b{\beta}}K_{i\b{j}\alpha\b{\beta}}\frac{v^{i}\b{v}^{j}}{G})\omega^{r'-1}_{FS}\wedge\omega^{n}
    -\frac{1}{2\pi n}\int_{P(E')/M}(g^{\alpha\b{\beta}}T_{i\b{j}\alpha\b{\beta}}\frac{v^{i}\b{v}^{j}}{G})\omega^{r'-1}_{FS}\wedge\omega^{n}\\
    &=\frac{\lambda}{2\pi n}\omega^{n}-\frac{1}{2\pi n}\int_{P(E')/M}(g^{\alpha\b{\beta}}T_{i\b{j}\alpha\b{\beta}}\frac{v^{i}\b{v}^{j}}{G})\omega^{r'-1}_{FS}\wedge\omega^{n}\\
   &=\frac{\mathcal{C}_{1}(E,G)\wedge\omega^{n-1}}{r}-\frac{1}{2\pi n}\int_{P(E')/M}(g^{\alpha\b{\beta}}T_{i\b{j}\alpha\b{\beta}}\frac{v^{i}\b{v}^{j}}{G})\omega^{r-1}_{FS}\wedge\omega^{n}.
    \end{split}
  \end{align*}
So we have proved that for any holomorphic subbundle $E'$ of $E$,
\begin{align}\label{4.b20}
\mu(E')\leq\mu(E).
\end{align}
\end{proof}
\begin{lemma}\label{4.33}  Let $(E,G)\to (M,\omega)$ be a Finsler-Einstein vector bundle over a closed K\"{a}hler manifold $M$ with the K\"{a}hler form $\omega$ and let $(L, h)$ be a Hermitian-Einstein line bundle with respects to $\omega$. Then
  $(E\otimes L, \tilde{G}=G\otimes h)$ is also a Finsler-Einstein vector bundle with respects to $\omega$.
\end{lemma}
\begin{proof} Clearly, $ \tilde{G}=G\otimes h$ is a complex Finsler metric on $E\otimes L$. Note that $P(E\otimes L)$ is biholomorphic equivalent to $P(E)$ naturally. Let $h^G$ and $h^{\tilde G}$ are the induced Hermitian metrics on $\pi^*E$ and $\pi^*(E\otimes L)$, respectively.
Then one has
\begin{align}\label{5.808}
\b\p\left((\p h^{\tilde G})(h^{\tilde G})^{-1}\right)=\b\p\left((\p h^G)(h^G)^{-1}+\p\log h\right)=\b\p((\p h^G)(h^G)^{-1})+\b\p\p\log h.
\end{align}
Since $E$ is Finsler-Einstein and $L$ is Hermitian-Einstein, by (\ref{5.808}) and (\ref{4.10}), (\ref{1.8}),  the proposition follows. Moreover,
If $\lambda_{E\otimes L}$, $\lambda_E$ and $\lambda_L$ are constants determined by the metric $\tilde G$, $G$ and $h$, respectively, then one verifies easily that
\begin{align}\label{5.909}
\lambda_{E\otimes L}=\lambda_E+\lambda_L.
\end{align}
\end{proof}
From the two lemmas above, we can prove the following
\begin{thm}\label{3ppp}
 Let $E$ be a Finsler-Einstein vector bundle over a closed K\"{a}hler manifold $M$ with the K\"{a}hler form $\omega$. If moreover, for each $m$ with $1<m<r$,  $\wedge^mE$ is also a Finsler-Einstein vector bundle over $(M,\omega)$, then $E$ is $\omega$-semi-stable.
\end{thm}
\begin{proof}
 For any subsheaf $\mc{F}$ of $\mc{E}=\mc{O}(E)$ of rank $m<r$ such that $\mc{E}/\mc{F}$ is torsion-free, there is a non-trivial holomorphic section of the bundle $\w^m E\otimes(\det\mc{F})^*$ (cf. \cite{Ko3}, page 178). Note that every line bundle admits an Hermitian-Einstein structure, we can choose an Hermitian-Einstein metric on $\det\mc{F}$ with constant
 $$\lambda'=\frac{2n\pi\cdot\mu(\det\mc{F})}{\int_M \omega^n}=\frac{2n m\pi\cdot\mu(\mc{F})}{\int_M \omega^n}.$$
 For $1<m<r$, let $G_m$ denote the Finsler-Einstein metric on $\w^mE\to (M,\omega)$. Then the corresponding constant $\lambda_m$ is given by
  \begin{align*}
    \lambda_m&=\frac{2n\pi}{\int_M\omega^n}\mu(\w^m E)\\
    &=\frac{2n\pi}{\int_M\omega^n}\frac{\int_M c_1(\w^m E)\wedge \omega^{n-1}}{{\rm rank}(\w^m E)}\\
    &=\frac{2n\pi}{\int_M\omega^n}\frac{\int_M \binom{r-1}{m-1}c_1(E)\wedge\omega^{n-1}}{\binom{r}{m}}\\
    &=\frac{2n m\pi}{\int_M\omega^n}\mu(E).
  \end{align*}
  By Lemma \ref{4.33} and (\ref{5.909}), the vector bundle $\w^m E\otimes (\det\mc{F})^*$ is a Finsler-Einstein vector bundle with factor $\lambda_m-\lambda'$. Since the bundle admits a non-trivial section, by Theorem 7.1 in \cite{Ko1}, we have
  $$\lambda_m-\lambda'\geq 0.$$
  This is nothing but the desired inequality $\mu(\mc{F})\leq \mu(\mc{E})$. This proves that $E$ is $\omega$-semi-stable.
\end{proof}
\begin{rem}
How to prove that a Finsler-Einstein vector bundle is semi-stable without the extra assumption seems a difficult problem. One may follow the method of Kobayashi in Chap. V in  \cite{Ko3}. However, this approach meets a difficult extension problem. On the other hand, how to equip $\w^mE$ with Finsler metrics from the one on $E$ turns out to be also a difficult question.
\end{rem}

Let $E$ be a holomorphic vector bundle over closed K\"{a}hler manifold $(M,\omega)$ with an Hermitian-Einstein metric $h$. Then one has the following Kobayashi-L\"{u}bke inequality (cf. \cite{Ko3}, \cite{Lubke}), i.e.
$$\int_{M}\left((r-1)c_{1}(E,h)^{2}-2rc_{2}(E,h)\right)\wedge \omega^{n-2}\leq 0,$$
where $\omega$ is the K\"{a}hler form on $M$, the equality holds if and only if $(E,h)$ is projectively flat (i.e., $R^{i}_{j\alpha\b{\beta}}=(1/r)\delta^{i}_{j}R_{\alpha\b{\beta}}$, (cf. \cite{Ko3}, page 7)).

In \cite{Diverio}, S. Diverio proved an inequality for Segre forms of Hermitian-Einstein vector bundles, from which the Kobayashi-L\"{u}bke inequality follows by considering the Hermtian-Einstein vector bundle $(E\otimes E^{*}, h\otimes \b{h}^{-1})$.
In the Finsler-Einstein case, we have
\begin{thm}\label{t4.9999}  If $(E,G)$ is a Finsler-Einstein vector bundle respect to $\omega$ , then
   \begin{align}\label{4.d11}
     s_{2}(E,G)\w\omega^{n-2}\leq \frac{r(r+1)}{8\pi^{2}n^{2}}\lambda^{2}\omega^n\quad\mbox{\rm on}\ M.
   \end{align}
Moreover, the equality holds if and only if $\Psi=\frac{\lambda}{n}\omega$.
\end{thm}
\begin{proof} From (\ref{2.15}), one has
\begin{align}
\begin{split}
  s_{2}(E,G)\w\omega^{n-2}&=\frac{r(r+1)}{8\pi^{2}}\int_{P(E)/M}\Psi^{2}\omega^{r-1}_{FS}\w\omega^{n-2}\\
  &=\frac{r(r+1)}{8\pi^{2}n(n-1)}\left[\int_{P(E)/M}((tr_{\omega}\Psi)^2-tr_{\omega}\Psi^2)\omega^{r-1}_{FS}\right]\omega^{n}\\
  &\leq\frac{r(r+1)}{8\pi^{2}n(n-1)}\left[\int_{P(E)/M}((tr_{\omega}\Psi)^2-\frac{1}{n}(tr_{\omega}\Psi)^2)\omega^{r-1}_{FS}\right]\omega^{n}\\
  &=\frac{r(r+1)}{8\pi^{2}n^{2}}\lambda^{2}\omega^{n},
  \end{split}
\end{align}
where $tr_{\omega}\Psi^2$ is defined as
\begin{align}\label{4.d12}
tr_{\omega}\Psi^2:=(K_{i\b{j}\alpha\b{\beta}}\frac{v^{i}\b{v}^{j}}{G})(K_{k\b{l}\gamma\b{\delta}}
\frac{v^{k}\b{v}^{l}}{G})g^{\alpha\b{\delta}}g^{\gamma\b{\beta}}.
\end{align}
Moreover, $n tr_{\omega}\Psi^{2}=(tr_{\omega}\Psi)^2$ if and only if
$$K_{i\b{j}\alpha\b{\beta}}\frac{v^{i}\b{v}^{j}}{G}=\frac{\lambda}{n}g_{\alpha\b{\beta}},$$
that is, $\Psi=\frac{\lambda}{n}\omega$. The proof is completed.
\end{proof}

Note that Deverio's trick fails to give the Kobayashi-L\"{u}bke inequality in Finsler-Einstein case, since it is very hard to construct an easily handled
Finsler metric on $E^*$ from the one on $E$. So we will give a direct proof of the following theorem
  \begin{thm}\label{4.18}
    If $G$ is a Finsler-Einstein metric, then
    $$((r-1)\mathcal{C}_{1}(E,G)^{2}-2r\mathcal{C}_{2}(E,G))\w\omega^{n-2}\leq 0$$
    at every point of $M$. The equality holds if and only if
    $\Psi=\frac{2\pi}{r}\mathcal{C}_{1}(E,G).$
  \end{thm}
  \begin{proof}
From (\ref{2.16}), a direct computation implies that
\begin{align}\label{4.e}
\begin{split}
&((r-1)\mathcal{C}_{1}(E,G)^{2}-2r\mathcal{C}_{2}(E,G))\w\omega^{n-2}\\
&=\frac{r+1}{n(n-1)}\omega^{n}\w\left[tr_{\omega}s_{1}(E,G)^{2}
  -\int_{P(E)/M}(tr_{\omega}(\frac{r}{2\pi}\Psi)^{2}\w
  \omega^{r-1}_{FS})\right].
\end{split}
\end{align}
Set $(\psi_{i\b{j}})=\frac{r}{2\pi}\Psi$. Since the equality (\ref{4.e}) holds pointwise, we can do computations in a normal coordinate system near $x\in M$ with $g_{\alpha\b{\beta}}=\delta_{\alpha\beta}$ and we have
  \begin{align*}
   &tr_{\omega}s_{1}(E,G)^{2}-\int_{P(E)/M}(tr_{\omega}(\frac{r}{2\pi}\Psi)^{2}\w\omega^{r-1}_{FS})\\
  &=\sum_{i,j}\int_{P(E)/M}(\psi_{i\b{j}}\omega^{r-1}_{FS})\int_{P(E)/M}(\o{\psi_{i\b{j}}}\omega^{r-1}_{FS})
  -\int_{P(E)/M}(\sum_{ij}|\psi_{i\b{j}}|^{2}\omega^{r-1}_{FS})\\
  &\leq(\int_{P(E)/M}(|\psi_{i\b{j}}|\omega^{r-1}_{FS}))^{2}-\int_{P(E)/M}(|\psi_{i\b{j}}|^{2}\omega^{r-1}_{FS})\leq 0,
  \end{align*}
moreover, the equality holds if and only if $(\psi_{i\b{j}})$ is constant along the fiber. By (\ref{4.11}), this s equivalent to
  $$\Psi=\frac{2\pi}{r}\mathcal{C}_{1}(E,G).$$
\end{proof}

From the above Kobayashi-L\"{u}bke inequality we have
\begin{cor}
Let $(E, G)\to (M,\omega)$ be a Finsler-Einstein vector bundle with respect to $\omega$. Assume that $\omega$ is closed and $\dim M=2$, ${\rm rank }(E)\geq 2$. Then $c_{2}(E,G)$ is a positive $(2,2)$-form if $G$ is of the positive or the negative Kobayashi curvature.
\end{cor}
\begin{proof} Under the assumption of this corollary, we know that $\mathcal{C}_{1}(E,G)$ is a positive or negative $(1,1)$-form by Theorem \ref{t6.2}. So $\mathcal{C}_{1}(E,G)^2$ is a positive $(2,2)$-form. Now the above Kobayashi-L\"{u}bke inequality implies that
 \begin{align*}
    \mathcal{C}_{2}(E,G)&\geq \frac{2r}{r-1}\mathcal{C}_{1}(E,G)^{2}>0.
  \end{align*}
\end{proof}

\section{Finsler flat vector bundles}

 An Hermitian metric $h$ on $E$ is flat if its Chern curvature vanishes. In this section, we introduce a definition of a flat Finsler metric which generalizes the flat Hermitian metric to the Finsler case. Moreover, by using a result of Berndtsson \cite{Ber}, we prove that a holomorphic vector bundle is Finsler flat if and only if it is Hermitian flat.

We make the following definition of a flat Finsler metric:
\begin{defn}\label{5.1} A strongly pseudo-convex Finsler metric $G$ on a holomorphic vector bundle $E$ is said to be \emph{flat} if its Kobayashi curvature $\Psi$ vanishes. A holomorphic vector bundle $E$ is called Finsler flat if it admits a flat Finsler metric.
\end{defn}

In local coordinates, this means $\Psi=K_{i\bar j \alpha\bar\beta}{{v^i\bar v^j}\over G}dz^\alpha\wedge d\bar z^\beta=0$.
In \cite{Aikou1}, Aikou  also introduced a definition of the flatness of a Finsler metric $G$. His definition actually requires that all
$K_{i\bar j \alpha\bar\beta}=0$. Clearly, Aikou's definition is stronger than ours.

Let $\pi:E\to M$ be a holomorphic vector bundle over a closed complex manifold $M$. It is known that the direct image $\{H^{0}(P(E_z),\mc{O}_{P(E_z)}(1))|z\in M\}\to M$ of $\mc{O}_{P(E)}(1)$ can be identified naturally with the dual bundle $E^*$, that is, for any $z\in M$, any vector $u\in E_z^*$ can be viewed as a holomorphic section of the bundle $\mc{O}_{P(E_z)}(1)\to P(E_z)$.

From this point-view and following Berndtsson \cite{Ber}, $E^{*}$ can be equipped with the following $L^2$-metric $h$
\begin{align}\label{5.2}
 h_z(u):=\int_{P(E_z)}|u|^{2}e^{-\phi}\omega^\phi_{r-1},\ z\in M,
\end{align}
where $u\in E^*_z\equiv H^{0}(P(E_{z}),\mc{O}_{P(E_z)}(1))$, which can be viewed as a holomorphic section of $\mc{O}_{P(E_z)}(1)$; and $\phi=\log G$, $\omega^\phi=\sqrt{-1}\p_{v}\b{\p}_{v}\phi$, $\omega^\phi_{r-1}=\frac{\m^{r-1}}{(r-1)!}$.

Let $\nabla$ denote the Chern connection on $(E^*, \|\bullet\|)$ associated to the metric $h$ on $E^*$ and let $\Theta^{E^{*}}$ be the curvature of the connection $\nabla$. As a special case of Theorem 3.1 in \cite{Ber}, we have
  \begin{align}\label{1}
  \sqrt{-1}\langle \Theta^{E^{*}}u,u\rangle\leq \frac{1}{(r-1)!}\int_{P(E)/M}\left(|u|^{2}e^{-\phi}(\omega^{\phi})^{r}\right).
  \end{align}
Actually, we can prove (\ref{1}) only for holomorphic sections $u$ of $E^*$ with $\n'u=0$ at any given point $z$. Note that $\n=\n'+\b{\p}$, one has
  \begin{align}\label{5.2222}
  \begin{split}
    \sqrt{-1}\langle\Theta^{E^{*}}u,u\rangle_z&=-(\sqrt{-1}\p\b{\p}\|u\|^{2})_{z}\\
    &=-\int_{P(E)/M}(\sqrt{-1}\p^{\phi}u\w \o{\p^{\phi}u}\w \omega^{\phi}_{r-1}e^{-\phi})+\int_{P(E)/M}(|u|^{2}\omega^{\phi}\w\omega^{\phi}_{r-1} e^{-\phi})\\
 &\leq \frac{1}{(r-1)!}\int_{P(E)/M}|u|^{2}e^{-\phi}(\omega^{\phi})^{r}.
  \end{split}
  \end{align}
  where $\p^{\phi}=e^{\phi}\p e^{-\phi}$.

Now we have the following theorem
\begin{thm}
If $E$ is a holomorphic vector bundle over a closed complex manifold $M$, then  $E$ admits a flat Finsler metric $G$ if and only if it admits a flat Hermitian structure.
\end{thm}
\begin{proof} A flat Hermitiant metric is obviously a flat Finsler metric. For any flat Finsler metric on $E$, one has the Kobayashi curvature $\Psi={\sqrt{-1}}K_{i\bar{j}\alpha\bar{\beta}}\frac{v^{i}\bar{v}^{j}}{G}dz^{\alpha}\w d\b{z}^{\beta}=0.$ So
\begin{align}\label{5.222}
\mathcal{C}_1(E,G)=-\int_{P(E)/M}c_1(\mc{O}_{P(E)}(1),h^G)^r={r\over {2\pi}}\int_{P(E)/M}\Psi\wedge\omega_{FS}^{r-1}=0.
\end{align}
On the other hand, by (\ref{5.2222}) one has
  \begin{align}\label{5.223}
  \sqrt{-1}\langle \Theta^{E^{*}}u,u\rangle_z\leq
  -r\int_{P(E_z)}|u|^{2}K_{i\bar{j}\alpha\bar{\beta}}\frac{v^{i}\bar{v}^{j}}{G}e^{-\phi}\omega_{r-1}dz^{\alpha}\w d\b{z}^{\beta}=0.
  \end{align}
So $(E,h^*)$ is Griffiths semi-positive, where $h^*$ is the dual $L^2$-metric on $E$. Let $R$ denote the Chern curvature of the metric $h^*$. Thus one has
  \begin{align}\label{5.226}
    c_{1}(E,h^*)=\frac{\sqrt{-1}}{2\pi}\text{Tr}(R)\geq 0,
  \end{align}
Combining (\ref{5.222}) and (\ref{5.226}), one gets
\begin{align}\label{5.227}
 c_{1}(E,h^*)=\sqrt{-1}\p\b{\p}\phi\geq 0.
 \end{align}
 Now from (\ref{5.227}) and note that the manifold is closed, we know that the function $\phi$ is constant. So we get $c_{1}(E,h^*)=0$. Again since  $h^*$ is Griffith semi-positive, the holomorphic bisectional curvature of $h^*$ is  zero and  the Chern connection of $h^*$ is flat.

\end{proof}


\begin{thebibliography}{99}

\bibitem{AP} M. Abate, G. Patrizio, Finsler Metrics- A Global Approach, LNM 1591, Springer-Verlag, Berlin Heidelberg, 1994.

\bibitem{Aikou1} T. Aikou, \textsl{Complex Finsler Geometry}, in Hand Book of Finsler Geometry (edited by P. Antonelli), 3-79, Kluwer , New York, 2003.

\bibitem{Aikou} T. Aikou, \textsl{Finsler Geometry on complex vector bundles}, Riemann-Finsler Geometry, MSRI Pulblications, Vol.50(2004), 83-105.

\bibitem{Ber} B. Berndtsson, \textsl{Positive of Direct image bundles and convexity on the space of K\"{a}hler metrics}, J. Differential Geometry. 81 (2009), 457-482.

\bibitem{Bo-C} R. Bott,  S. S. Chern,  \textsl{Hermitian vector bundles and the equidistribution of the zeroes of
their holomorphic sections}, Acta Math. 114 (1965), 71-112.

\bibitem{BT} R. Bott, L. W. Tu, Differential Forms in Algebraic Topology, Springer.

\bibitem{BG} S. Bloch, D. Gieseker, \textsl{The positivity of the Chern classes of an ample vector bundle}, Inventiones Math. 12 (1971), 112-117.

\bibitem{Cao-Wong} J. Cao, Pit-Mann Wong, \textsl{Finsler geometry of projectivized vector bundles}, J. Math. Kyoto Univ, 2003.

\bibitem{Diverio} S. Diverio, \textsl{An inequality for segre forms of Hermtian-Einstein vector bundles}, arXiv:1503.02512v1.

\bibitem{Faran} J. Faran, \textsl{The equivalence problem for complex Finsler Hamiltonians}, in Finsler Geometry (Seattle, WA,
1995), Contemp. Math., Vol. 196, Amer. Math. Soc., Providence, RI, 1996, 133¨C144.

\bibitem{Fulton} W. Fulton, R. Lazarsfeld, \textsl{Positive polynomail for ample vector bundles},  Annals of Math.  118 (1983), 35-60.

\bibitem{Guler} D. Guler, \textsl{On Segre forms of positive vector bundles}. Canad. Math. Bull. 55(1) (2012), 108-113.

\bibitem{Gr-Har} P. Griffith, J. Harris, Principles of Algebraic Geometry, Wiley, New York, 1978.

\bibitem{Hart} R. Hartshorne, \textsl{Ample vector bundle}, Publ. Math. IHES. tome 29(1966), 63-94.

\bibitem{IDA} C. IDA, \textsl{Horizontal Forms of Chern Type on Complex Finsler Bundles}, SIGMA 6 (2010), 054, 1-7.

\bibitem{Ko1} S. Kobayashi, \textsl{Negative vector bundles and complex Finsler structures}, Nagoya Math. J. Vol. 57 (1975), 153-166.

\bibitem{Ko2} S. Kobayashi, \textsl{Complex Finsler vector bundles}, Contemp. Math., Vol. 196, Amer. Math. Soc., Providence, RI, 1996, 133¨C144.

\bibitem{Ko3} S. Kobayashi, Differential Geometry of Complex Vector Bundles, Iwanami-Princeton Univ. Press, 1987.

\bibitem{Lubke} M. L\"{u}bke, \textsl{Chernklassen von Hermitian-Einstein-Vektorb\"{u}ndeln}, Math. Ann. 260 (1982), 133-141.

\bibitem{Mour}  C. Mourougane, \textsl{ Computations of Bott-Chern classes on P(E)}, Duke Math. J.  124  (2004),  no. 2, 389¨C420.

\bibitem{Sun} L. Sun, \textsl{On Complex Finsler-Einstein vector bundles (in Chinese)}, Ph.D Dissertation, Xiamen University, 2014.

\bibitem{Zhang} W. Zhang, Lectures on Chern-Weil Theory and Witten Deformations, Nankai Tracts in Mathematics, Vol.4, World Scientific Publishing Co. Pte. Ltd., 2001.


\end{thebibliography}
\end{document}